\documentclass{amsart}

\usepackage{a4wide}
\usepackage[initials]{amsrefs}
\usepackage{url}
\usepackage{color}
\usepackage[utf8]{inputenc}
\usepackage{amsmath,amssymb,amsthm}
\usepackage{float}
\usepackage{enumerate}

\newtheorem{lem}{Lemma}
\newtheorem*{theo*}{Theorem}

\newtheorem{cor}{Corollary}

\newtheorem{innercustomlem}{Lemma}
\newenvironment{customlem}[1]
  {\renewcommand\theinnercustomlem{#1}\innercustomlem}
  {\endinnercustomlem}

\title{On Erd\H{o}s and S\'{a}rközy's sequences with Property P}

\author[C. Elsholtz]{Christian Elsholtz}
\address[C. Elsholtz]{
Graz University of Technology, Institute of Analysis and Number Theory, Kopernikusgasse 24/II, 8010 Graz, Austria}
\email{elsholtz@math.tugraz.at}

\author[S. Planitzer]{Stefan Planitzer}
\address[S. Planitzer]{
Graz University of Technology, Institute of Analysis and Number Theory, Kopernikusgasse 24/II, 8010 Graz, Austria}
\email{planitzer@math.tugraz.at}



\begin{document}

\begin{abstract}
A sequence $A$ of positive integers having the property that no element $a_i \in A$ divides the sum $a_j+a_k$ of two larger elements is said to have `Property P'. We construct an infinite set $S\subset \mathbb{N}$ having Property P with counting function $S(x)\gg\frac{\sqrt{x}}{\sqrt{\log x}(\log\log x )^2(\log \log \log x)^2}$. This improves on an example given by Erd\H{o}s and S\'{a}rközy with a lower bound on the counting function of order $\frac{\sqrt{x}}{\log x}$.
\end{abstract}

\maketitle

\section{Introduction}

Erd\H{o}s and S\'{a}rközy~\cite{DivisibilityProperties} define a monotonically increasing sequence $A=\{a_1<a_2< \ldots\}$ of positive integers to have `Property P' if $a_i \nmid a_j+a_k$ for $i<j<k$. They proved that any infinite sequence of integers with Property P has density $0$. Schoen~\cite{OnA} showed that if an infinite sequence $A$ has Property P and any two elements in $A$ are coprime then the counting function $A(x)=\sum_{a_i<x}{1}$ is bounded from above by $A(x)<2x^{\frac{2}{3}}$ and Baier~\cite{ANote} improved this to $A(x)<(3+\epsilon)x^{\frac{2}{3}}(\log x)^{-1}$ for any $\epsilon >0$. Concerning finite sequences with Property P, Erd\H{o}s and S\'{a}rközy~\cite{DivisibilityProperties} get the lower bound $\max A(x) \geq \lfloor \frac{x}{3} \rfloor+1$ by just taking $A$ to be the set $A=\{x, x -1, \ldots, x  - \lfloor \frac{x}{3} \rfloor \}$ for $x \in \mathbb{N}$.   

Erd\H{o}s and S\'{a}rközy also thought about large sets with Property P with respect to the size of the counting function (cf.~\cite[p. 98]{DivisibilityProperties}). They observed that the set $A=\{q_i^2: q_i \text{ the }i\text{-th prime with }q_i \equiv 3 \bmod 4\}$ has Property P. This uses the fact that the square of a prime $p \equiv 3 \bmod 4$ has only the trivial representation $p^2=p^2+0^2$ as the sum of two squares. With this set $A$ they get
$$A(x) \sim \frac{\sqrt{x}}{\log x}.$$
Erd\H{o}s has asked repeatedly to improve this (see e.g.~\cite[p. 185]{SomeOld},~\cite[p. 535]{SomeOfMy}) and in particular, Erd\H{o}s~\cite{SomeOfMy,SomeOf} asked if one can do better than $a_n \sim (2n \log n)^2$. He wanted to know if it is possible to have $a_n < n^2$. We will not quite achieve this but we go a considerable step in this direction. First, we observe that a set of squares of integers consisting of precisely $k$ prime factors $p \equiv 3 \bmod 4$ also has property P. As for any fixed $k$ this would only lead to a moderate improvement, our next idea is to try to choose $k$ increasing with $x$. In order to do so, we actually use a union of several sets $S_i$ with property P. Together, this union will have a good counting function throughout all ranges of $x$. However, in order to ensure that this union of sets with property P still has property P, we employ a third idea, namely to equip all members $a\in S_i$ with a special indicator factor. This seems to be the first improvement going well beyond the example given by Erd\H{o}s and S\'{a}rközy since $1970$. Our main result will be the following theorem.
\begin{theo*} \label{MainTheorem}
The set $S \subset \mathbb{N}$ constructed explicitly below has Property P and counting function
$$S(x) \gg \frac{\sqrt{x}}{\sqrt{\log x}(\log\log x )^2(\log \log \log x)^2}.$$
\end{theo*}
We achieve this improvement by not only considering squares of primes $p \equiv 3 \bmod 4$ but products of squares of such primes. More formally we set 
\begin{equation} \label{SDefinition}
S=\bigcup_{i=1}^{\infty}{S_i}.
\end{equation}
Here the sets $S_i$ are defined by
\begin{equation} \label{SIDefinition}
S_i:=\left\{n\in \mathbb{N}: n=q_i^4\nu^2\right\},
\end{equation}
where $\nu$ is the product of exactly $i$ distinct primes $p \equiv 3 \bmod 4$ and we recall that $q_i$ is the $i$-th prime in the residue class $3 \bmod 4$. The rôle of the $q_i$ is an `indicator' which uniquely identifies the set $S_i$ a given integer $n \in S$ belongs to. Results from probabilistic number theory like the Theorem of Erd\H{o}s-Kac suggest that for varying $x$ different sets $S_i$ will yield the main contribution to the counting function $S(x)$. In particular for given $x>0$ the main contribution comes from the sets $S_i$ with
$$\frac{\log \log \sqrt{x}}{2}-\sqrt{\frac{\log \log \sqrt{x}}{2}} \leq i \leq \frac{\log \log \sqrt{x}}{2}+\sqrt{\frac{\log \log \sqrt{x}}{2}}.$$ 

The study of sequences with Property P is closely related to the study of primitive sequences, i.e. sequences where no element divides any other and there is a rich literature on this topic (cf. \cite[Chapter V]{Sequences}). Indeed a similar idea as the one described above was used by Martin and Pomerance~\cite{PrimitiveSets} to construct a large primitive set. While Besicovitch~\cite{OnThe} proved that there exist infinite primitive sequences with positive upper density, Erd\H{o}s~\cite{NoteOn} showed that the lower density of these sequences is always $0$. In his proof Erd\H{o}s used the fact that for a primitive sequence of positive integers the sum $\sum_{i=1}^{\infty}{\frac{1}{a_i\log a_i}}$ converges. In more recent work Banks and Martin~\cite{OptimalPrimitive} make some progress towards a conjecture of Erd\H{o}s which states that in the case of a primitive sequence 
$$\sum_{i=1}^{\infty}{\frac{1}{a_i\log a_i}} \leq \sum_{p \in \mathbb{P}}{\frac{1}{p \log p}}$$ 
holds. Erd\H{o}s~\cite{OnSequences} studied a variant of the Property P problem, also in its multiplicative form.

\section{Notation}

Before we go into details concerning the proof of the Theorem we need to fix some notation. Throughout this paper $\mathbb{P}$ denotes the set of primes and the letter $p$ (with or without index) will always denote a prime number. We write $\log_k$ for the $k$-fold iterated logarithm. The functions $\omega$ and $\Omega$ count, as usual, the prime divisors of a positive integer $n$ without respectively with multiplicity. For two functions $f,g : \mathbb{R} \rightarrow \mathbb{R}^+$ the binary relation $f \gg g$ (and analogously $f \ll g$) denotes that there exists a constant $c > 0$ such that for $x$ sufficiently large $f(x) \geq cg(x)$ ($f(x) \leq cg(x)$ respectively). Dependence of the implied constant on certain parameters is indicated by subscripts. The same convention is used for the Landau symbol $\mathcal{O}$ where $f=\mathcal{O}(g)$ is equivalent to $f \ll g$. We write $f=o(g)$ if $\lim_{x\rightarrow \infty} \frac{f(x)}{g(x)}=0$.

\section{The set $S$ has Property P}

In this section we verify that any union of sets $S_i$ defined in (\ref{SIDefinition}) has Property P.

\begin{lem} \label{DivisibilityLemma}
Let $n_1,n_2$ and $n_3$ be positive integers. If there exists a prime $p \equiv 3 \bmod 4$ with $p | n_1$ and $p \nmid \gcd (n_2,n_3)$, then
$$n_1^2\nmid n_2^2+n_3^2.$$
\end{lem}

\begin{proof}
We prove the Lemma by contradiction.  Suppose that $n_1^2|n_2^2+n_3^2$.
By our assumption there exists a prime $p \equiv 3 \bmod 4$ such that   $p|n_1$ and $p \nmid \gcd(n_2,n_3)$. Hence, w.l.o.g. $p \nmid n_2$. 
We have
$$n_2^2+n_3^2 \equiv 0 \bmod p$$
and since $p$ does not divide $n_2$, we get that $n_2$ is invertible $\bmod$ $p$. Hence
$$\left(\frac{n_3}{n_2}\right)^2\equiv -1 \bmod p$$
a contradiction since $-1$ is a quadratic non-residue $\bmod$ $p$.
\end{proof}

\begin{lem} \label{SHasPropertyP}
Any union of sets $S_i$ defined in (\ref{SIDefinition}) has Property P.
\end{lem}

\begin{proof}
Suppose by contradiction that there exist $a_i \in S_i$, $a_j\in S_j$ and $a_k \in S_k$ with $a_i<a_j<a_k$ and $a_i|a_j+a_k$. First suppose that either $i \neq j$ or $i \neq k$. Define $l \in \{0,2\}$ to be the largest exponent such that $q_i^l|\gcd(a_i,a_j,a_k)$ where we again recall that $q_i$ was defined as the $i$-th prime in the residue class $3 \bmod 4$. Then 
$$\frac{a_i}{q_i^l}\bigg|\frac{a_j}{q_i^l}+\frac{a_k}{q_i^l}.$$
By construction of the sets $S_i,S_j$ and $S_k$ we have that $q_i\big|\frac{a_i}{q_i^l}$ and w.l.o.g. $q_i \nmid \frac{a_j}{q_i^l}$. An application of Lemma \ref{DivisibilityLemma} finishes this case. 

If $i=j=k$ then $\Omega(a_i)=\Omega(a_j)=\Omega(a_k)$. If there is some prime $p$ with $p|a_i$ and ($p\nmid a_j$ or $p \nmid a_k$) we may again use Lemma \ref{DivisibilityLemma}. If no such $p$ exists, then $a_i|a_j$ and $a_i|a_k$ trivially holds. With the restriction on the number of prime factors we get that $a_i=a_j=a_k$.
\end{proof}

\section{Products of $k$ distinct primes}

In order to establish a lower bound for the counting functions of the sets $S_i$ in (\ref{SIDefinition}) we need to count square-free integers containing exactly $k$ distinct prime factors $p \equiv 3 \bmod 4$, but no others, where $k \in \mathbb{N}$ is fixed. For $k \geq 2$ and $\pi_k(x):=\#\{n \leq x: \omega(n)=\Omega(n)=k\}$ Landau~\cite{SurQuelques} proved the following asymptotic formula:
\begin{equation*} 
\pi_k(x) \sim \frac{x(\log_2 x)^{k-1}}{(k-1)!\log x}.
\end{equation*}
We will need a lower bound of similar asymptotic growth as the formula above for the quantity 
$$\pi_k(x;4,3):=\#\{n \leq x: p|n \Rightarrow p \equiv 3 \bmod 4,\text{ } \omega(n)=\Omega(n)=k\}.$$
Very recently Meng~\cite{LargeBias} used tools from analytic number theory to prove a generalization of this result to square-free integers having $k$ prime factors in prescribed residue classes. The following is contained as a special case in~\cite[Lemma 9]{LargeBias}:
\begin{customlem}A \label{MengLemma}
For any $A>0$, uniformly for $2 \leq k \leq A\log \log x$, we have
\begin{align*}
\pi_k(x;4,3)=&\frac{1}{2^k}\frac{x}{\log x}\frac{(\log \log x)^{k-1}}{(k-1)!} \times \\
&\left(1+\frac{k-1}{\log\log x}C(3,4)+\frac{2(k-1)(k-2)}{(\log \log x)^2}h''\left(\frac{2(k-3)}{3 \log \log x}\right)+\mathcal{O}_{A}\left(\frac{k^2}{(\log \log x)^3}\right)\right),
\end{align*}
where $C(3,4)=\gamma+\sum_{p \in \mathbb{P}}\left(\log\left(1-\frac{1}{p}\right)+\frac{2\lambda(p)}{p}\right)$, $\gamma$ is the Euler-Mascheroni constant, $\lambda(p)$ is the indicator function of primes in the residue class $3 \bmod 4$ and 
$$h(x)=\frac{1}{\Gamma \left(\frac{x}{2}+1\right)}\prod_{p\in \mathbb{P}}\left(1-\frac{1}{p}\right)^{x/2}\left(1+\frac{x\lambda(p)}{p}\right).$$
\end{customlem}
We will show that Lemma \ref{MengLemma} with some extra work implies the following Corollary.
\begin{cor} \label{PiKLowerBound}
Uniformly for $\frac{\log \log x}{2}-1 \leq k \leq \frac{\log \log x}{2}+\sqrt{\frac{\log \log x}{2}}$ we have
$$\pi_k(x;4,3) \gg \frac{1}{2^k}\frac{x}{\log x}\frac{(\log_2x)^{k-1}}{(k-1)!}.$$
\end{cor}
\begin{proof}
In view of Lemma \ref{MengLemma} and with $k\sim\frac{\log \log x}{2}$ we see that it suffices to check that, independent of the choice of $k$ and for sufficiently large $x$, there exists a constant $c>0$ such that 
\begin{equation} \label{PiKLowerBound c definition}
1+\frac{C(3,4)}{2}+\frac{1}{2}h''\left(\frac{2(k-3)}{3 \log \log x}\right) \geq c.
\end{equation}
Note that the left hand side of the above inequality is exactly the coefficient of the main term $\frac{1}{2^k}\frac{x}{\log x}\frac{(\log_2x)^{k-1}}{(k-1)!}$ for $k$ in the range given in the Corollary. The constant $C(3,4)$ does not depend on $k$. Using Mertens' Formula (cf. \cite[p. 19: Theorem 1.12]{IntroductionTo}) in the form
$$\sum_{\substack{p \in \mathbb{P} \\ p \leq x}}\log\left(1-\frac{1}{p}\right)=-\gamma-\log\log x+ o(1)$$
we get
$$C(3,4)=\gamma+\sum_{p\in \mathbb{P}}\left(\log\left(1-\frac{1}{p}\right)+\frac{2\lambda(p)}{p}\right)=2M(3,4),$$
where $M(3,4)$ is the constant appearing in
$$\sum_{p \in \mathbb{P}}\frac{\lambda (p)}{p} = \frac{\log \log x}{2}+M(3,4)+\mathcal{O}\left(\frac{1}{\log x}\right),$$
which was studied by Languasco and Zaccagnini in~\cite{ComputingThe}\footnote{Note that our constant $M(3,4)$ corresponds to the constant $M(4,3)$ in the work of Languasco and Zaccagnini.}. The computational results of Languasco and Zaccagnini imply that $0.0482<M(3,4)<0.0483$ and hence allow for the following lower bound for $C(3,4)$:
\begin{equation} \label{PiKLowerBound c34 bound}
C(3,4) =2M(3,4)>0.0964.
\end{equation}
It remains to get a lower bound for $h''\left(\frac{2(k-3)}{3 \log \log x}\right)$, where the function $h$ is defined as in Lemma~\ref{MengLemma}. A straight forward calculation yields that
$$h'=\prod_{p \in \mathbb{P}}\left(1-\frac{1}{p}\right)^{x/2}\left(1+\frac{x\lambda(p)}{p}\right)\frac{\Gamma\left(\frac{x}{2}+1\right)\left(\sum_{p \in \mathbb{P}}\frac{1}{2}\log\left(1-\frac{1}{p}\right)+\frac{\lambda(p)}{p+x\lambda(p)}\right)-\frac{1}{2}\Gamma'\left(\frac{x}{2}+1\right)}{\Gamma\left(\frac{x}{2}+1\right)^2}$$
and
$$h''(x)=f(x)\prod_{p \in \mathbb{P}}\left(1-\frac{1}{p}\right)^{x/2}\left(1+\frac{x\lambda(p)}{p}\right),$$
where
\begin{align*}
f(x)&=\frac{\left(\sum_{p \in \mathbb{P}}\frac{1}{2}\log\left(1-\frac{1}{p}\right)+\frac{\lambda(p)}{p+x\lambda(p)}\right)^2}{\Gamma\left(\frac{x}{2}+1\right)}-\frac{\Gamma''\left(\frac{x}{2}+1\right)}{4\Gamma\left(\frac{x}{2}+1\right)^2}-\frac{\sum_{p \in \mathbb{P}}\frac{\lambda(p)}{(p+\lambda(p)x)^2}}{\Gamma\left(\frac{x}{2}+1\right)}\\
&-\frac{\Gamma'\left(\frac{x}{2}+1\right)\left(\sum_{p \in \mathbb{P}}\frac{1}{2}\log\left(1-\frac{1}{p}\right)+\frac{\lambda(p)}{p+x\lambda(p)}\right)}{\Gamma\left(\frac{x}{2}+1\right)^2}+\frac{\Gamma'\left(\frac{x}{2}+1\right)^2}{2\Gamma\left(\frac{x}{2}+1\right)^3}.
\end{align*}
Note that for $x \rightarrow \infty$ and $\frac{\log \log x}{2}-1 \leq k \leq \frac{\log \log x}{2}+\sqrt{\frac{\log \log x}{2}}$ the term $\frac{2(k-3)}{3 \log \log x}$ gets arbitrarily close to $\frac{1}{3}$. Hence we may suppose that $\frac{99}{300}\leq \frac{2(k-3)}{3 \log \log x} \leq \frac{101}{300}$ and it suffices to find a lower bound for $h''(x)$ where $\frac{99}{300}\leq x \leq \frac{101}{300}$. For $x$ in this range Mathematica provides the following bounds on the Gamma function and its derivatives
$$0.9271 \leq \Gamma\left(\frac{x}{2}+1 \right) \leq 0.9283 \text{, } -0.3104 \leq \Gamma'\left(\frac{x}{2}+1 \right) \leq -0.3058 \text{, } 1.3209 \leq \Gamma''\left(\frac{x}{2}+1 \right) \leq 1.3302.$$
Furthermore we have
$$\sum_{p\in \mathbb{P}}\frac{\lambda(p)}{(p+x)^2}<\sum_{p\in \mathbb{P}}\frac{\lambda(p)}{p^2}<\sum_{\substack{p \in \mathbb{P} \\ p \leq 10^4}}\frac{\lambda(p)}{p^2}+\sum_{n>10^4}\frac{1}{n^2}<0.1485+\int_{x=10^4}^{\infty}\frac{\mathrm{d}x}{x^2}=0.1486.$$
Later we will use that
\begin{align*}
\sum_{p \in \mathbb{P}}\left(\frac{1}{2}\log\left(1-\frac{1}{p}\right)+\frac{\lambda(p)}{p+x}\right)&=\sum_{p \in \mathbb{P}}\left(\frac{1}{2}\log\left(1-\frac{1}{p}\right)+\frac{\lambda(p)}{p}\right)-x\sum_{p \in \mathbb{P}}\frac{\lambda(p)}{p^2+px} \\
&>\sum_{p \in \mathbb{P}}\left(\frac{1}{2}\log\left(1-\frac{1}{p}\right)+\frac{\lambda(p)}{p}\right)-x\sum_{p \in \mathbb{P}}\frac{\lambda(p)}{p^2} \\
&=-\frac{\gamma}{2}+M(3,4)-x\sum_{p \in \mathbb{P}}\frac{\lambda(p)}{p^2}>-0.2905,
\end{align*}
and
$$\sum_{p \in \mathbb{P}}\left(\frac{1}{2}\log\left(1-\frac{1}{p}\right)+\frac{\lambda(p)}{p+x}\right) <\sum_{p \in \mathbb{P}}\left(\frac{1}{2}\log\left(1-\frac{1}{p}\right)+\frac{\lambda(p)}{p}\right)=-\frac{\gamma}{2}+M(3,4)<-0.2403.$$
Finally, using $\log(1+\frac{x}{p})\leq \frac{x}{p}$, we get
\begin{align*}
0&\leq \prod_{p \in \mathbb{P}}\left(1-\frac{1}{p}\right)^{x/2}\left(1+\frac{x\lambda(p)}{p}\right)\leq \exp\left(x\left(\sum_{p \in \mathbb{P}}\left(\frac{1}{2}\log\left(1-\frac{1}{p}\right)+\frac{\lambda(p)}{p}\right)\right)\right) \\
&=\exp\left(x\left(-\frac{\gamma}{2}+M(3,4)\right)\right)<\exp\left(-\frac{99}{300}\cdot 0.2403\right)<0.9238.
\end{align*}
Applying the explicit bounds calculated above, for $\frac{99}{300} \leq x \leq \frac{101}{300}$ we obtain:
$$f(x) \geq \frac{0.2403^2}{0.9283}-\frac{1.3302}{4\cdot 0.9271^2}-\frac{0.1486}{0.9271} -\frac{0.3104 \cdot 0.2905}{0.9271^2}+\frac{0.3058^2}{2\cdot 0.9283^3} > -0.5315.$$
which implies for sufficiently large $x$:
$$h''\left(\frac{2(k-3)}{3 \log \log x}\right) > -0.492.$$
Together with \eqref{PiKLowerBound c34 bound} this leads to an admissible choice of $c=0.802$ in \eqref{PiKLowerBound c definition}.
\end{proof}

\section{The counting function $S(x)$}

\begin{proof}[Proof of Theorem]
As in (\ref{SDefinition})  we set 
$$S= \bigcup_{i=1}^{\infty}S_i$$
where the sets $S_i$ are defined as in (\ref{SIDefinition}). The set $S$ has Property P by Lemma \ref{SHasPropertyP} and it remains to work out a lower bound for the size of the counting function $S(x)$. For sufficiently large $x$ there exists a uniquely determined integer $k \in \mathbb{N}$ such that $e^{2e^{2k}} \leq x  < e^{2e^{2(k+1)}}$ hence 
\begin{equation} \label{iInequality}
k \leq \frac{\log_2 \sqrt{x}}{2} < k+1. 
\end{equation}
It depends on the size of $x$, which $S_i$ makes the largest contribution. For a given $x$ we take several sets $S_{k+2},S_{k+3}, \ldots, S_{k+l}$, $l = \lfloor\sqrt{\frac{\log_2 \sqrt{x}}{2}} \rfloor$, as the number of prime factors $p \equiv 3 \bmod 4$ of a typical integer less than $x$ is in 
$$\left[\frac{\log_2 x}{2} - \sqrt{\frac{\log_2 x}{2}},\frac{\log_2 x}{2} + \sqrt{\frac{\log_2 x}{2}}\right].$$ 
Using Corollary \ref{PiKLowerBound} as well as the fact that the $i$-th prime in the residue class $3 \bmod 4$ is asymptotically of size $2 i \log i$ for given $2 \leq j \leq l$ we get

\begin{equation} \label{SiInequality}
S_{k+j}(x) \gg \underbrace{\frac{\sqrt{\frac{x}{16(k+j)^4\log^4(k+j)}}}{ \log\left(\sqrt{\frac{x}{16(k+j)^4\log^4(k+j)}}\right)}}_{\mathrm{F}_1} \cdot \underbrace{\frac{\left(\log_2 \sqrt{\frac{x}{16(k+j)^4\log^4(k+j)}}\right)^{k+j-1}}{2^{k+j} (k+j-1)!}}_{\mathrm{F}_2}.
\end{equation}
We deal with the fractions $\mathrm{F}_1$ and $\mathrm{F}_2$ on the right hand side of (\ref{SiInequality}) separately. With the given range of $j$ and (\ref{iInequality}) we have that
$$\mathrm{F}_1 \gg \frac{\sqrt{x}}{\log x (\log_2 x)^2 (\log_3x)^2}.$$
It remains to deal with $\mathrm{F}_2$. Using the given range of $k$ and $j$ we have that $k+j \leq \log_2\sqrt{x}$ and, again for sufficiently large $x$, for the numerator of $F_2$ we get
\begin{align*}
\log_2^{k+j-1} \sqrt{\frac{x}{16(k+j)^4\log^4(k+j)}} &\gg (\log(\log \sqrt{x}-\log 4 -2\log_3\sqrt{x}-2\log_4 \sqrt{x}))^{k+j-1} \\
&\gg (\log (\log \sqrt{x}-5\log_3 \sqrt{x}))^{k+j-1} \\
&=  \left(\log_2\sqrt{x}+\log\left(1-\frac{5\log_3\sqrt{x}}{\log \sqrt{x}}\right)\right)^{k+j-1} \\
&\gg \left(\log_2 \sqrt{x}-\frac{10\log_3 \sqrt{x}}{\log \sqrt{x}}\right)^{k+j-1} \\
&\gg\left(1-\frac{10 \log_3\sqrt{x}}{\log \sqrt{x} \log_2\sqrt{x}}\right)^{\frac{\log_2 \sqrt{x}}{2}+\sqrt{\frac{\log_2 \sqrt{x}}{2}}-1}\log_2^{k+j-1}\sqrt{x}\\
&\gg\log_2^{k+j-1}\sqrt{x}.
\end{align*}
Here we used that
$$\lim_{x \rightarrow \infty}\left(1-\frac{10 \log_3\sqrt{x}}{\log \sqrt{x} \log_2\sqrt{x}}\right)^{\frac{\log_2 \sqrt{x}}{2}+\sqrt{\frac{\log_2 \sqrt{x}}{2}}-1} = 1$$
and that for $0 \leq y \leq \frac{1}{2}$ we certainly have that $\log(1-y) \geq -2y$.
To deal with the denominator of $\mathrm{F}_2$ we apply Stirling's Formula and get
\begin{align*}
(k+j-1)! &\ll \left(\frac{k+j-1}{e}\right)^{k+j-1}\sqrt{k+j-1} \ll \left( \frac{\log_2 \sqrt{x}+2(j-1)}{2e}\right)^{k+j-1}\sqrt{\log_2x} \\
&\ll (\log_2 \sqrt{x}+2(j-1))^{k+j-1}\frac{\sqrt{\log_2x}}{2^{k+j-1}e^{\frac{\log_2\sqrt{x}}{2}+j-2}} \\
&\ll (\log_2 \sqrt{x}+2(j-1))^{k+j-1}\frac{\sqrt{\log_2x}}{2^{k+j-1}e^{j-2}\sqrt{\log x}}.
\end{align*}
Altogether we get 
\begin{equation} \label{SikEquation}
\begin{split}
\mathrm{F}_2 &\gg \frac{\sqrt{\log x}}{\sqrt{\log_2 x}} e^{j-2}\left(\frac{\log_2 \sqrt{x}}{\log_2 \sqrt{x}+2(j-1)}\right)^{k+j-1} \\
&\gg \frac{\sqrt{\log x}}{\sqrt{\log_2 x}} e^{j-2}\left(\frac{\log_2 \sqrt{x}}{\log_2 \sqrt{x}+2(j-1)}\right)^{\frac{\log_2 \sqrt{x}}{2}+j-1}.
\end{split}
\end{equation}
Since
$$\left(\frac{\log_2 \sqrt{x}}{\log_2 \sqrt{x}+2(j-1)}\right)^{\frac{\log_2 \sqrt{x}}{2}} \sim \frac{1}{e^{j-1}}$$
it suffices to check that for any $x>0$ and for our choices of $j$ there exists a fixed constant $c>0$ such that
\begin{equation} \label{cInequality}
\left(1+\frac{2(j-1)}{\log_2\sqrt{x}}\right)^{1-j}\geq c.
\end{equation}
For $j\geq 2$ we have that $\left(1+\frac{2(j-1)}{\log_2\sqrt{x}}\right)^{1-j}$ is monotonically decreasing in $j$ and get
$$\left(1+\frac{2(j-1)}{\log_2\sqrt{x}}\right)^{1-j} \geq \left(1+\frac{2\sqrt{\frac{\log_2 \sqrt{x}}{2}}}{\log_2\sqrt{x}}\right)^{-\sqrt{\frac{\log_2 \sqrt{x}}{2}}}=\left(1+\frac{1}{\sqrt{\frac{\log_2\sqrt{x}}{2}}}\right)^{-\sqrt{\frac{\log_2 \sqrt{x}}{2}}}\geq\frac{1}{e}.$$
Therefore for $j \geq 2$ the constant $c$ in (\ref{cInequality}) may be chosen as $c=\frac{1}{e}$ for sufficiently large $x$. Together with (\ref{SikEquation}) this implies
$$F_2 \gg \frac{\sqrt{\log x}}{\sqrt{\log_2x}}.$$
Altogether for the counting function of any of the sets $S_i$ with $\lfloor \frac{\log_2\sqrt{x}}{2}\rfloor+2 \leq i \leq \lfloor \frac{\log_2\sqrt{x}}{2}\rfloor + \lfloor \sqrt{\frac{\log_2\sqrt{x}}{2}}\rfloor$ we have 
$$S_i(x) \gg \frac{\sqrt{x}}{\sqrt{\log x}(\log_2 x)^{\frac{5}{2}}(\log_3 x)^2}.$$
Summing these contributions up we finally get
$$S(x) \gg \frac{\sqrt{x}}{\sqrt{\log x}(\log_2 x)^2(\log_3 x)^2}.$$
\end{proof}

\section*{Acknowledgement}
The authors are supported by the Austrian Science Fund (FWF): W1230, Doctoral Program `Discrete Mathematics'. Parts of this research work were done when the first author was visiting the FIM at ETH Z\"urich, and the second author was
visiting the Institut \'Elie Cartan de Lorraine of the University of
Lorraine. The authors thank these institutions for their hospitality. The authors are also grateful to the referee for suggestions on the manuscript and would like to thank Xianchang Meng for some discussion on his recent paper \cite{LargeBias}.


\end{document}